\documentclass[12pt]{article}
\usepackage{amsmath,amssymb,amsthm}
\usepackage{booktabs}
\usepackage{hyperref}
\usepackage{indentfirst}

\title{The Linear Slicing Method for Equal Sums of Like Powers: Modular and Geometric Constraints}

\author{Valery Asiryan\\[3pt]
\small \texttt{asiryanvalery@gmail.com}}
\date{\small December 6, 2025}

\theoremstyle{plain}
\newtheorem{theorem}{Theorem}[section]
\newtheorem{lemma}[theorem]{Lemma}
\newtheorem{corollary}[theorem]{Corollary}
\newtheorem{proposition}[theorem]{Proposition}

\theoremstyle{remark}
\newtheorem{remark}[theorem]{Remark}

\begin{document}

\maketitle

\begin{abstract}
We study the classical Diophantine equation
\[
a^k + b^k = c^k + d^k,
\]
with non-negative integer variables $a,b,c,d\in\mathbb{Z}_{\ge0}$ and real exponent $k>1$, under the additional linear constraint
\[
(c+d) - (a+b) = h, \qquad h \in \mathbb{Z}.
\]
We view this as a ``linear slicing'' of the surface of equal sums of two $k$th powers by the planes
$a+b=S$ and $c+d=S+h$, and we analyze the geometry and arithmetic of these slices.

On the central slice $h=0$, we show, using strict convexity, that there are
no non-trivial solutions even over the reals: if $a,b,c,d\ge0$ and $a+b=c+d$, then
$a^k+b^k=c^k+d^k$ forces $\{a,b\}=\{c,d\}$.  For integer $k\ge2$ we also prove a quantitative
separation theorem on the hyperplane $a+b=c+d$, showing that distinct unordered pairs with the same sum
$S$ produce values of $a^k+b^k$ that are separated by $\gg_k S^{k-2}$.

Our main new input for shifted slices $h\neq0$ is a \emph{modular divisibility obstruction (MDO)}:
for every integer $k\ge2$ define
\[
\mathcal{P}_k := \{\,p\ \text{prime} : (p-1)\mid (k-1)\,\},\qquad
M_k := \prod_{p\in\mathcal{P}_k} p.
\]
Then any integer solution with shift $h$ must satisfy $M_k\mid h$; this modulus $M_k$ is the
\emph{maximal squarefree} modulus for which $x^k\equiv x\pmod{M_k}$ holds for all integers $x$.
As immediate corollaries, if $k$ is prime then $k\mid h$, and for fixed $k$ the necessary condition
$M_k\mid h$ leaves only a $1/M_k$ fraction of shifts.
In particular, this creates a divisibility filter that eliminates, for instance, $99.96\%$ of all linear slices for $k=13$ (since $M_{13}=2730$).
Combining MDO with our convexity analysis yields an exclusion zone principle. While this implies a \emph{combined bound} on slice size $\min\{S,S+h\} \ge 2|h|/(k-1)$, we show that a global overlap argument provides a strictly stronger constraint.
Finally, along any fixed slice $(S,h)$ we prove an \emph{asymptotic dominance} bound
$k\le \max\{S,S+h\}\log 2$, beyond which no integer solutions can occur.

We also formulate the central-slice uniqueness and separation results for arbitrary
strictly convex functions, and we discuss how these structural restrictions fit with probabilistic
spacing heuristics and with the Bombieri--Lang philosophy.



\end{abstract}

\section{Introduction}

The Diophantine equation
\begin{equation}\label{eq:base}
a^k + b^k = c^k + d^k, \qquad a,b,c,d \in \mathbb{Z}_{\ge0},
\end{equation}
has a long history.  For $k=2$ there are classical parametrizations of all integer solutions.
For $k=3$ and $k=4$ there exist non-trivial parametrizations as well, leading to famous examples such as
\[
1729 = 1^3 + 12^3 = 9^3 + 10^3
\]
and
\[
635318657 = 59^4 + 158^4 = 133^4 + 134^4,
\]
which are the first taxicab numbers of orders $3$ and $4$, respectively; see for instance the survey
of Lander--Parkin--Selfridge~\cite{LanderParkinSelfridge1967}.

For $k \ge 5$, however, the situation is dramatically different.
As of December~2025, no non-trivial integer solutions to \eqref{eq:base} are known in this range, and the existence of such solutions remains a difficult and largely open problem;
see, for example, Browning~\cite{Browning2002} and the references therein, as well as Guy's collection of unsolved problems~\cite{Guy2004}.

\medskip

In this paper we impose the linear constraint
\begin{equation}\label{eq:shift}
(c+d) - (a+b) = h,\qquad h\in\mathbb{Z},
\end{equation}
and regard~\eqref{eq:base} as a family of equations indexed by the \emph{shift parameter} $h$.
Introducing
\[
S := a+b,\qquad S+h := c+d,
\]
we compare the functions
\[
f_S(x) := x^k + (S-x)^k, \qquad f_{S+h}(y) := y^k + (S+h-y)^k,
\]
along the \emph{linear slices} $a+b=S$ and $c+d=S+h$; we refer to this viewpoint as the
\emph{linear slicing method}.

The central slice $h=0$ corresponds to the hyperplane $a+b=c+d$.  There, strict convexity immediately
forces uniqueness up to permutation, and in the integer case we obtain a quantitative separation of order
$S^{k-2}$ between distinct values of $a^k+b^k$ along $a+b=S$.

\medskip

Our main new contribution concerns shifted slices $h\neq0$:
\begin{itemize}
\item a \emph{modular divisibility obstruction (MDO)}: $M_k\mid h$ with
$M_k=\prod_{p:\,p-1\mid (k-1)}p$, the maximal squarefree modulus with $x^k\equiv x\pmod{M_k}$ for all $x$;
this contains parity as the case $p=2$ and yields strong filters for many odd $k$ (e.g.\ $k=13\Rightarrow M_k=2730$);
\item for integer $k\ge2$, an \emph{exclusion zone} of radius $\asymp \sqrt{\min\{S,S+h\}}$ around $\min\{S,S+h\}/2$ on the
smaller-sum slice for each fixed $h\neq0$;
\item an \emph{asymptotic dominance} bound along any fixed slice $(S,h)$: writing $S_0=\max\{S,S+h\}$ and
$M=\max\{a,b,c,d\}$ for a solution, one must have $k\le M\log 2\le S_0\log 2$.
\end{itemize}
For integer solutions, combining MDO with the exclusion zone gives the lower bound $\min\{S,S+h\}\ge 2M_k/(k-1)$ whenever $h\neq0$.

\medskip

\noindent\textbf{Organization of the paper.}
Section~\ref{sec:central} develops the theory of the central slice $h=0$, including qualitative
uniqueness and quantitative separation on $a+b=c+d$.
Section~\ref{sec:general-convex} formulates central-slice uniqueness for general strictly convex functions.
Section~\ref{sec:modular} proves the modular divisibility obstruction (MDO), its maximality, and corollaries
(parity as a special case, density, prime exponents), and includes a table of $M_k$ for odd $k$.
Section~\ref{sec:shifted-geometry} contains the exclusion zone principle and its combination with MDO,
yielding $\min\{S,S+h\}\ge 2M_k/(k-1)$ for $h\neq0$, together with a separate ``overlap'' bound coming from the
global ranges of $a^k+b^k$ on different slices.
Section~\ref{sec:asymptotic-dominance} establishes the asymptotic dominance bounds for fixed slices.
Finally, Section~\ref{sec:conclusion} discusses heuristics and broader context.

Throughout, $\mathbb{N}=\{0,1,2,\dots\}$ and, unless explicitly stated otherwise, we take $a,b,c,d\in\mathbb{Z}_{\ge0}$ and use $\log$ to denote the natural logarithm.

\section{The central slice $h=0$: uniqueness and quantitative separation}
\label{sec:central}

We begin with the most symmetric slice, $a+b=c+d$.

\subsection{Qualitative uniqueness on the central slice}

\begin{theorem}\label{thm:central-uniqueness}
Let $k>1$ be real and let $S\ge 0$.
Define $f_S(x) = x^k + (S-x)^k$ on $[0,S]$.
Then $f_S$ is strictly decreasing on $[0,S/2]$ and strictly increasing on $[S/2,S]$.
In particular, if $x_1,x_2\in[0,S]$ and $f_S(x_1)=f_S(x_2)$, then
$\{x_1,S-x_1\}=\{x_2,S-x_2\}$.
\end{theorem}

\begin{proof}
We have
$f_S'(x)=k\bigl(x^{k-1}-(S-x)^{k-1}\bigr)$ and
$f_S''(x)=k(k-1)\bigl(x^{k-2}+(S-x)^{k-2}\bigr)>0$ on $(0,S)$,
so $f_S$ is strictly convex and symmetric with $f_S'(S/2)=0$.
Hence it is strictly decreasing on $[0,S/2]$ and strictly increasing on $[S/2,S]$.
Injectivity on $[0,S/2]$ gives the conclusion.
\end{proof}

\begin{corollary}\label{cor:central-real}
If $k>1$ and $a,b,c,d\ge0$ satisfy $a^k+b^k=c^k+d^k$ and $a+b=c+d$, then $\{a,b\}=\{c,d\}$.
\end{corollary}

\begin{corollary}\label{cor:central-diophantine}
If $k>1$ and $a,b,c,d\in\mathbb{Z}_{\ge0}$ satisfy $a^k+b^k=c^k+d^k$ and $a+b=c+d$,
then $\{a,b\}=\{c,d\}$.
\end{corollary}

\begin{corollary}\label{cor:two-powers}
Let $k>m>0$ be real and $a,b,c,d\ge0$ satisfy $a^m+b^m=c^m+d^m$ and $a^k+b^k=c^k+d^k$.
Then $\{a,b\}=\{c,d\}$.
\end{corollary}

\subsection{Quantitative separation on the central slice}

\begin{lemma}\label{lem:discrete-convex}
If $\varphi$ is strictly convex on an interval and $x_n=x_0+nh$ lie in that interval, then
$u_n:=\varphi(x_n)$ satisfies $u_{n+1}-u_n$ strictly increasing in $n$.
\end{lemma}

\begin{proof}
Strict convexity gives $\varphi(x_n)<\frac12(\varphi(x_{n-1})+\varphi(x_{n+1}))$,
whence $u_{n+1}-u_n>u_n-u_{n-1}$.
\end{proof}

\begin{theorem}\label{thm:quantitative}
Let $k\ge2$ be integer and let $a,b,c,d\in\mathbb{Z}_{\ge0}$ satisfy
$a+b=c+d=S\ge2$ and $\{a,b\}\neq\{c,d\}$.
Then
\[
\bigl|a^k+b^k - c^k-d^k\bigr|
\;\ge\; k(k-1)\,\bigl\lfloor S/2\bigr\rfloor^{\,k-2}
\;\ge\; C_k\, S^{k-2},\qquad C_k=k(k-1)\,3^{2-k}.
\]
\end{theorem}

\begin{proof}
Let $S\ge2$ and put $n:=\lfloor S/2\rfloor$.
For $t\in\{0,1,\dots,S\}$ set $F(t):=t^k+(S-t)^k$.
Since $F(t)=F(S-t)$, every unordered pair $\{u,S-u\}$ is represented uniquely by
$x:=\min\{u,S-u\}\in\{0,1,\dots,n\}$, and its value is $F(x)$.
By Theorem~\ref{thm:central-uniqueness}, $F$ is strictly decreasing on $[0,S/2]$,
hence the sequence $u_t:=F(t)$ for $t=0,1,\dots,n$ is strictly decreasing.

Apply Lemma~\ref{lem:discrete-convex} to $f_S$ on the arithmetic progression $0,1,\dots,n$.
Then $\Delta_t:=u_{t+1}-u_t$ is strictly increasing in $t$, so
$g_t:=u_t-u_{t+1}=-(\Delta_t)$ is strictly decreasing in $t$.
Therefore for any $0\le x<y\le n$ we have
\[
u_x-u_y=\sum_{t=x}^{y-1} g_t \ \ge\ g_{n-1}=u_{n-1}-u_n.
\]
Thus, for distinct unordered pairs with sum $S$,
\[
|a^k+b^k-c^k-d^k| \ge u_{n-1}-u_n.
\]

It remains to bound $u_{n-1}-u_n$ from below.

\smallskip
Case $S=2n$ (even).
Then $u_n=2n^k$ and $u_{n-1}=(n-1)^k+(n+1)^k$, hence
\[
\begin{aligned}
u_{n-1}-u_n
&=(n+1)^k+(n-1)^k-2n^k \\
&= 2\sum_{\substack{j\ \mathrm{even}\\ j\ge 2}} \binom{k}{j} n^{k-j} \\
&\ge k(k-1)n^{k-2}.
\end{aligned}
\]

\smallskip\
Case $S=2n+1$ (odd).
Then $u_n=n^k+(n+1)^k$ and $u_{n-1}=(n-1)^k+(n+2)^k$, so
\[
u_{n-1}-u_n=(n+2)^k-(n+1)^k+(n-1)^k-n^k
=\sum_{j=2}^k \binom{k}{j} n^{k-j}\bigl(2^j-1+(-1)^j\bigr).
\]
Here each coefficient $\bigl(2^j-1+(-1)^j\bigr)\ge0$ and equals $4$ when $j=2$, hence
\[
u_{n-1}-u_n \ge 4\binom{k}{2}n^{k-2}=2k(k-1)n^{k-2}\ge k(k-1)n^{k-2}.
\]

Combining the cases yields
\[
|a^k+b^k-c^k-d^k|\ge k(k-1)\,n^{k-2}=k(k-1)\,\lfloor S/2\rfloor^{\,k-2}.
\]
Finally, for $S\ge2$ one has $\lfloor S/2\rfloor\ge S/3$, hence
$\lfloor S/2\rfloor^{k-2}\ge 3^{2-k}S^{k-2}$, giving the constant $C_k$.
\end{proof}

\section{A general formulation for strictly convex functions}
\label{sec:general-convex}

\begin{theorem}\label{thm:general}
Let $I\subset\mathbb{R}$ be an interval and $\varphi:I\to\mathbb{R}$ strictly convex.
Fix $S\in\mathbb{R}$ and define $F(x)=\varphi(x)+\varphi(S-x)$ on
$D=I\cap(S-I)$.
Then $F$ is strictly decreasing on $D\cap(-\infty,S/2]$, strictly increasing on
$D\cap[S/2,\infty)$, and $F(x_1)=F(x_2)$ implies $\{x_1,S-x_1\}=\{x_2,S-x_2\}$.
\end{theorem}

\begin{proof}
$F$ is strictly convex by convexity of $\varphi$ and symmetric: $F(x)=F(S-x)$.
Hence it has a unique minimum at $S/2$ and the claimed monotonicity/uniqueness follow.
For background, cf.\ Karamata/majorization~\cite{HLP1934,NiculescuPersson2006}.
\end{proof}

\section{Modular constraints on the shift: the MDO}\label{sec:modular}

We derive a universal necessary congruence for the shift $h$ depending only on $k$.

\begin{theorem}[Modular Divisibility Obstruction]\label{thm:MDO}
Let $k\ge2$ be integer, set
\[
\mathcal{P}_k=\{\,p\ \mathrm{prime}:\ (p-1)\mid(k-1)\,\},\qquad
M_k=\prod_{p\in\mathcal{P}_k}p.
\]
If $a,b,c,d\in\mathbb{Z}$ satisfy $a^k+b^k=c^k+d^k$ and $h=(c+d)-(a+b)$, then $M_k\mid h$.
\end{theorem}

\begin{proof}
Fix $p\in\mathcal{P}_k$, so $k-1=m(p-1)$ for some $m\in\mathbb{N}$. For any integer $x$,
either $x\equiv0\pmod p$ (then $x^k\equiv x\equiv0$) or $x\not\equiv0\pmod p$, in which case by FLT
$x^{p-1}\equiv1\pmod p$ and $x^k=x(x^{p-1})^m\equiv x\pmod p$.
Hence $a^k+b^k\equiv a+b\pmod p$ and $c^k+d^k\equiv c+d\pmod p$,
so $a+b\equiv c+d\pmod p$ and thus $p\mid h$.
Since this holds for each $p\in\mathcal{P}_k$ and the primes are coprime, $M_k\mid h$ by CRT~\cite{IrelandRosen1990}.
\end{proof}

\begin{lemma}[Maximality on squarefree moduli]\label{lem:maximality}
Let $N$ be squarefree and suppose $x^k\equiv x\pmod N$ holds for all integers $x$.
Then $N\mid M_k$.
\end{lemma}

\begin{proof}
If $p\mid N$, then $x^k\equiv x\pmod p$ for all $x$.
Restricting to units shows $x^{k-1}\equiv1$ for all $x\in(\mathbb{Z}/p\mathbb{Z})^\times$,
so the exponent of $(\mathbb{Z}/p\mathbb{Z})^\times$ divides $k-1$.
Since $(\mathbb{Z}/p\mathbb{Z})^\times$ is cyclic of order $p-1$, its exponent equals $p-1$.
Hence $p-1\mid k-1$, i.e.\ $p\in\mathcal{P}_k$,
and therefore $p\mid M_k$.
As $N$ is squarefree, $N\mid M_k$.
\end{proof}

\begin{corollary}[Parity as a special case]\label{cor:parity}
For every $k\ge2$, $2\in\mathcal{P}_k$ and hence $2\mid h$.
\end{corollary}

\begin{corollary}[Prime exponents]\label{cor:prime-k}
If $k$ is prime, then $k\in\mathcal{P}_k$ and thus $k\mid h$.
\end{corollary}

\begin{corollary}[Density of admissible shifts]\label{cor:density}
Among all integers $h$, the necessary condition $M_k\mid h$ selects a subset of asymptotic density $1/M_k$.
\end{corollary}

\begin{remark}[Arithmetic progressions in $k$ for fixed $h$]\label{rem:progressions}
Fix $h\ne0$ and a prime $p\nmid h$.
Then any $k$ with $(p-1)\mid(k-1)$ is \emph{forbidden} by Theorem~\ref{thm:MDO}.
In particular, if $3\nmid h$, all odd $k$ are excluded, since $2\mid(k-1)$ for odd $k$.
\end{remark}

\begin{remark}[Even exponents]\label{rem:even}
If $k$ is even, then $k-1$ is odd and the only prime $p$ with $(p-1)\mid(k-1)$ is $p=2$.
Hence $\mathcal{P}_k=\{2\}$ and $M_k=2$; i.e.\ MDO coincides with parity.
For clarity, our table below lists only \emph{odd} exponents.
\end{remark}

\subsection*{A compact table of $M_k$ for odd exponents}
For odd $k\in\{3,5,\dots,19\}$ we record $\mathcal{P}_k$, the squarefree modulus $M_k$,
the density $1/M_k$, and the combined lower bound on the slice size (see Proposition~\ref{prop:combination})
written as $2M_k/(k-1)$.

\medskip

\begin{center}
\begin{tabular}{r l r l l}
\toprule
$k$ & $\mathcal{P}_k=\{p:\ p-1\mid k-1\}$ & $M_k$ & density $1/M_k$ & $2M_k/(k-1)$ \\
\midrule
3  & $\{2,3\}$                 & $6$    & $1/6$     & $6$ \\
5  & $\{2,3,5\}$               & $30$   & $1/30$    & $15$ \\
7  & $\{2,3,7\}$               & $42$   & $1/42$    & $14$ \\
9  & $\{2,3,5\}$               & $30$   & $1/30$    & $60/8$ \\
11 & $\{2,3,11\}$              & $66$   & $1/66$    & $132/10$ \\
13 & $\{2,3,5,7,13\}$          & $2730$ & $1/2730$  & $5460/12=455$ \\
15 & $\{2,3\}$                 & $6$    & $1/6$     & $12/14$ \\
17 & $\{2,3,5,17\}$            & $510$  & $1/510$   & $1020/16$ \\
19 & $\{2,3,7,19\}$            & $798$  & $1/798$   & $1596/18$ \\
\bottomrule
\end{tabular}

\vspace{0.5ex}
\small\emph{Note.} Even $k$: $M_k=2$ (parity only), hence omitted.
\end{center}

\begin{remark}[Rapid growth for larger exponents]\label{rem:rapid-growth}
While the table lists $M_k$ for $k \le 19$, the value of $M_k$ grows very rapidly for exponents with highly composite $k-1$.
For instance, at $k=61$ we have $60=2^2\cdot3\cdot5$, leading to
\[
\mathcal P_{61}=\{2,3,5,7,11,13,31,61\},\qquad
M_{61}=2\cdot3\cdot5\cdot7\cdot11\cdot13\cdot31\cdot61=56{,}786{,}730.
\]
Thus any nonzero shifted solution must satisfy $|h|\ge M_{61}\approx 5.68\times 10^7$, i.e.\ the nearest admissible slice to $h=0$ is almost $57$ million away.
As a more dramatic illustration, for $k=841$ (so $k-1=840$) one finds many primes with $p-1\mid 840$, and even a partial product already exceeds $10^{21}$, pushing $|h|$ into the sextillion scale.
\end{remark}

\begin{remark}[Oscillation of admissible-shift densities]\label{rem:oscillation}
Recall from Corollary~\ref{cor:density} that for each exponent $k\ge2$ the condition $M_k\mid h$
selects a set of shifts of asymptotic density $d_k:=1/M_k$.
Since $M_k\ge2$ for all $k$ and $M_k=2$ for every even $k$ (Remark~\ref{rem:even}), we have
$d_k\le\tfrac12$ for all $k$ and $d_k=\tfrac12$ for infinitely many $k$, whence
\[
\limsup_{k\to\infty} d_k = \frac12.
\]
On the other hand, by choosing exponents with $k-1$ highly divisible (for example, taking $k-1$ to be a multiple of
$L_n=\mathrm{lcm}(1,2,\dots,n)$), we force $\mathcal{P}_k$ to contain all primes $p\le n+1$.
Thus
\[
M_k \;\ge\; \prod_{p\le n+1}p \;\longrightarrow\; \infty\quad\text{as }n\to\infty,
\]
so along such a sequence we have $d_k=1/M_k\to0$, and hence
\[
\liminf_{k\to\infty} d_k = 0.
\]
In the language of real analysis, the sequence $(d_k)$ therefore oscillates between values arbitrarily close to $1/2$ and values arbitrarily close to $0$.
\end{remark}

\section{Geometry of shifted slices and exclusion zones}
\label{sec:shifted-geometry}

Fix real $k>1$. For $S\ge0$ set $f_S(x)=x^k+(S-x)^k$ and
\[
V_{\min}(S):=\min_{x\in[0,S]} f_S(x)=2\Bigl(\frac{S}{2}\Bigr)^k.
\]

\begin{lemma}
$V_{\min}(S)$ is strictly increasing and strictly convex on $(0,\infty)$, with
\[
V_{\min}(S+h)-V_{\min}(S)\ \ge\ k\Bigl(\frac{S}{2}\Bigr)^{k-1}h\qquad(h>0).
\]
\end{lemma}

\begin{proof}
Direct differentiation gives
$V'_{\min}(S)=k(S/2)^{k-1}$ and $V''_{\min}(S)=\tfrac12k(k-1)(S/2)^{k-2}>0$.
Apply the mean value theorem and monotonicity of $V'_{\min}$.
\end{proof}

\begin{lemma}[Overlap bound for shifted slices]\label{lem:overlap}
Let $k>1$, and let $S>0$ and $h\in\mathbb{R}$ be such that $S+h>0$.
Suppose there exist real numbers $a,b,c,d\ge0$ with
\[
a+b=S,\quad c+d=S+h,\quad a^k+b^k=c^k+d^k.
\]
Then, after possibly interchanging $(a,b)$ with $(c,d)$, we may assume $h>0$ and $S\le S+h$, and in this case
\[
S \;\ge\; \frac{h}{2^{\frac{k-1}{k}}-1}.
\]
Equivalently, in symmetric form,
\[
\min\{S,S+h\} \;\ge\; \frac{|h|}{2^{\frac{k-1}{k}}-1}.
\]
\end{lemma}

\begin{proof}
By symmetry we may assume $h>0$ and $S\le S+h$.
On the slice $a+b=S$ the values of $a^k+b^k$ form the interval
\[
I_k(S) = \Bigl[\,2\Bigl(\frac S2\Bigr)^k,\ S^k\Bigr],
\]
and on the slice $c+d=S+h$ the values of $c^k+d^k$ form
\[
I_k(S+h) = \Bigl[\,2\Bigl(\frac{S+h}2\Bigr)^k,\ (S+h)^k\Bigr].
\]
If $a^k+b^k=c^k+d^k$, then $I_k(S)\cap I_k(S+h)\neq\varnothing$, so in particular
the left endpoint of $I_k(S+h)$ cannot exceed the right endpoint of $I_k(S)$:
\[
2\Bigl(\frac{S+h}{2}\Bigr)^k \;\le\; S^k.
\]
Rewriting,
\[
2\cdot\frac{(S+h)^k}{2^k} \le S^k
\quad\Longleftrightarrow\quad
\frac{(S+h)^k}{2^{k-1}} \le S^k
\quad\Longleftrightarrow\quad
\frac{S+h}{2^{\frac{k-1}{k}}} \le S.
\]
Thus
\[
h \le \bigl(2^{\tfrac{k-1}{k}}-1\bigr)S,
\]
which is the claimed inequality. The symmetric form follows by interchanging $S$ and $S+h$ when $h<0$.
\end{proof}

\begin{corollary}[The case $k=13$]\label{cor:overlap-13}
Let $k=13$, and let $a,b,c,d\ge0$ satisfy $a^{13}+b^{13}=c^{13}+d^{13}$ with
$S=a+b$, $S+h=c+d$ and $h\ne0$.
Then
\[
\min\{S,S+h\}\ \ge\ C_{13}\,|h|,\qquad
C_{13}:=\frac{1}{2^{12/13}-1}\approx 1.115878.
\]
In particular, for integer solutions, Theorem~\ref{thm:MDO} implies that $M_{13}=2730$ divides $h$,
so $|h|\ge M_{13}$ when $h\ne0$, and hence
\[
\min\{S,S+h\}\ \ge\ C_{13}M_{13} > 3046,
\]
i.e.\ $\min\{S,S+h\}\ge 3047$ on any non-central slice with $k=13$.
\end{corollary}

\begin{theorem}[Exclusion zone principle]\label{thm:exclusion}
Let $k\ge2$ be integer, $S>0$, and $h\ne0$.
Suppose $a,b,c,d\ge0$ and
$a^k+b^k=c^k+d^k$, $a+b=S$, $c+d=S+h$.
Let $\delta=|a-b|/2$.
If $h>0$ (so $S$ is the smaller sum), then
\[
\delta^2\ \ge\ \frac{S h}{2(k-1)}.
\]
The case $h<0$ is symmetric for $(c,d)$.
\end{theorem}

\begin{proof}
Let $X=S/2$. Since $a,b\ge 0$, we have $0\le\delta\le X$. We write $a=X-\delta, b=X+\delta$.
The necessary condition for a solution is $a^k+b^k \ge \min_{c+d=S+h} (c^k+d^k)$.
\[
(X-\delta)^k + (X+\delta)^k \ge 2\left(\frac{S+h}{2}\right)^k = 2(X+h/2)^k.
\]
We aim to show that this implies $(k-1)\delta^2 \ge Xh$.

We will prove the following inequality, which shows that the $k$th power sum is bounded above by the $k$th power of its quadratic approximation at the mean $X$:
\begin{equation}\label{eq:power_mean_inequality}
\frac{(X-\delta)^k+(X+\delta)^k}{2} \le \left(X + \frac{k-1}{2X}\delta^2\right)^k.
\end{equation}
If this inequality holds, then combining it with the necessary condition gives:
\[
2(X+h/2)^k \le (X-\delta)^k + (X+\delta)^k \le 2\left(X + \frac{k-1}{2X}\delta^2\right)^k.
\]
Since the function $t\mapsto t^k$ is strictly increasing for $t\ge 0$, taking the $k$th root yields:
\[
X+h/2 \le X + \frac{k-1}{2X}\delta^2,
\]
which simplifies to $Xh \le (k-1)\delta^2$, or $\delta^2 \ge \frac{Xh}{k-1} = \frac{Sh}{2(k-1)}$, as desired.

It remains to prove~\eqref{eq:power_mean_inequality}. We compare the binomial expansions of both sides.
The Left Hand Side (LHS) is:
\begin{align*}
\text{LHS} &= \frac{1}{2} \sum_{j=0}^k \binom{k}{j} X^{k-j} ((-\delta)^j + \delta^j) \\
&= \sum_{m=0}^{\lfloor k/2 \rfloor} \binom{k}{2m} X^{k-2m} \delta^{2m}.
\end{align*}
The Right Hand Side (RHS) is:
\begin{align*}
\text{RHS} &= \sum_{m=0}^k \binom{k}{m} X^{k-m} \left(\frac{k-1}{2X}\delta^2\right)^m \\
&= \sum_{m=0}^k \binom{k}{m} \left(\frac{k-1}{2}\right)^m X^{k-2m} \delta^{2m}.
\end{align*}
We will show that the inequality holds term by term, i.e., the coefficient of $X^{k-2m}\delta^{2m}$ on the LHS is less than or equal to the corresponding coefficient on the RHS for all $m\ge 0$. We need to verify:
\begin{equation}\label{eq:coeff_comparison}
\binom{k}{2m} \le \binom{k}{m} \left(\frac{k-1}{2}\right)^m, \quad \text{for } 1\le m \le \lfloor k/2 \rfloor.
\end{equation}
For $m=0$, both sides are 1. For $m > \lfloor k/2 \rfloor$, the LHS coefficient is 0, while RHS is non-negative.

\smallskip
Case $m=1$. LHS is $\binom{k}{2} = \frac{k(k-1)}{2}$. RHS is $\binom{k}{1}\frac{k-1}{2} = \frac{k(k-1)}{2}$. They are equal.

\smallskip
Case $m\ge 2$. We rewrite \eqref{eq:coeff_comparison} (assuming $2m\le k$):
\[
\frac{k!}{(2m)!(k-2m)!} \le \frac{k!}{m!(k-m)!} \left(\frac{k-1}{2}\right)^m
\]
\[
\frac{m!(k-m)!}{(2m)!(k-2m)!} \le \left(\frac{k-1}{2}\right)^m.
\]
The LHS is a product of $m$ factors:
\[
\text{LHS} = \frac{(k-m)(k-m-1)\cdots(k-2m+1)}{(2m)(2m-1)\cdots(m+1)} = \prod_{j=0}^{m-1} \frac{k-m-j}{2m-j}.
\]

Since $m\ge 2$ and $0\le j\le m-1$, we have $k-m-j\le k-2$ and $2m-j\ge m+1\ge 3$.
Therefore
\[
\frac{k-m-j}{2m-j}\le \frac{k-2}{3}\le \frac{k-1}{2},
\]
so each factor in the product is $\le (k-1)/2$. Hence
\[
\prod_{j=0}^{m-1} \frac{k-m-j}{2m-j}\le \left(\frac{k-1}{2}\right)^m,
\]
which proves \eqref{eq:coeff_comparison}.

We have shown that the expansions satisfy the inequality term by term, with equality for the $m=0$ and $m=1$ terms. This proves~\eqref{eq:power_mean_inequality} and completes the proof of the theorem.
\end{proof}

\begin{proposition}[Combination with MDO: lower bound on slice size]\label{prop:combination}
Under the hypotheses of Theorem~\ref{thm:exclusion}, one has
\[
\min\{S,S+h\}\ \ge\ \frac{2|h|}{k-1}.
\]
If moreover $h\ne0$ and $k\ge2$ is integer, then by Theorem~\ref{thm:MDO} $M_k\mid h$, hence
\[
\min\{S,S+h\}\ \ge\ \frac{2M_k}{k-1}.
\]
\end{proposition}

\begin{proof}
If $h>0$, then $\delta\le S/2$ and Theorem~\ref{thm:exclusion} gives
$S^2/4\ge \delta^2\ge Sh/(2(k-1))$, hence $S\ge 2h/(k-1)$.
If $h<0$, apply Theorem~\ref{thm:exclusion} to the pair $(c,d)$ with shift $-h>0$:
then $S+h \ge 2|h|/(k-1)$.
Combining, $\min\{S,S+h\}\ge 2|h|/(k-1)$.
The second bound follows since $M_k\mid h$ implies $|h|\ge M_k$ when $h\ne0$.
\end{proof}

\begin{remark}[Relative strength of bounds]
It is important to note that the overlap bound (Lemma~\ref{lem:overlap}) is strictly stronger than the combined bound (Proposition~\ref{prop:combination}) for all $k\ge 2$. That is,
\[
C_k := \frac{1}{2^{\frac{k-1}{k}}-1} > \frac{2}{k-1}.
\]
This inequality is equivalent to $k-1 > 2(2^{(k-1)/k}-1)$, or $k+1 > 4\cdot 2^{-1/k}$.
For $k=2$, $3 > 4\cdot 2^{-1/2} \approx 2.828$. For $k\ge 3$, $k+1\ge 4$, while $4\cdot 2^{-1/k} < 4$.
While the exclusion zone principle (Theorem~\ref{thm:exclusion}) provides insight into the local geometry near the center of the slice, the global constraint from the overlap of ranges (Lemma~\ref{lem:overlap}) is dominant.
\end{remark}

\section{Asymptotic dominance on fixed slices}
\label{sec:asymptotic-dominance}

We record a simple growth obstruction for large $k$ on a fixed slice. Let $\log$ be natural.

\begin{lemma}\label{lem:largest-gap}
If $M\ge2$ and $k> M\log 2$, then $M^k-(M-1)^k>(M-1)^k$, hence $M^k>2(M-1)^k$.
\end{lemma}

\begin{proof}
$\displaystyle \frac{M^k}{(M-1)^k}=\bigl(1+\frac{1}{M-1}\bigr)^k
=\exp\!\bigl(k\log(1+\tfrac{1}{M-1})\bigr)\ge \exp(k/M)$
since $\log(1+x)\ge x/(1+x)$ with $x=1/(M-1)$.
If $k/M>\log2$ the ratio exceeds $2$.
\end{proof}

\begin{theorem}[Asymptotic dominance on a fixed slice]\label{thm:asymptotic}
Let $a,b,c,d\in\mathbb{Z}_{\ge0}$, $h\ne0$, and $k\ge1$ satisfy
$a^k+b^k=c^k+d^k$ and $(c+d)-(a+b)=h$.
Let $S_1=a+b$, $S_2=c+d$, $S_0=\max\{S_1,S_2\}$, and $M=\max\{a,b,c,d\}$.
Then $M\ge2$ and
\[
k\ \le\ M\log 2\ \le\ S_0\log 2.
\]
In particular, fixing $S$ and $h$, there are no solutions with $a+b=S$, $c+d=S+h$ and
$k> K_0(S,h):=\max\{S,S+h\}\log 2$.
\end{theorem}

\begin{proof}
Since $h\ne0$, not all of $a,b,c,d$ lie in $\{0,1\}$; hence $M\ge2$.
If the maximal base $M$ appeared on both sides, then after reordering we may assume $a=c=M$.
Cancelling $M^k$ gives $b^k=d^k$, hence $b=d$, and therefore $h=(c+d)-(a+b)=0$, a contradiction.
Thus for $h\ne0$ the value $M$ appears on exactly one side; say $M^k+b^k=c^k+d^k$ with $b,c,d\le M-1$.
Then $M^k-(M-1)^k\le |d^k-b^k|\le (M-1)^k$, which fails when $k>M\log 2$ by Lemma~\ref{lem:largest-gap}.
Thus $k\le M\log 2\le S_0\log 2$.
\end{proof}

\noindent\emph{Link to MDO.} Thus, for a fixed slice, $k$ is bounded from above (Theorem~\ref{thm:asymptotic}), and for many $k$ within that range, the slice is forbidden by MDO (Theorem~\ref{thm:MDO}).

\begin{remark}
The bound is crude but explicit and sufficient for our local purposes; typically $M\approx S_0/2$,
suggesting a heuristic threshold near $(S_0/2)\log 2$.
\end{remark}

\section{Concluding remarks}\label{sec:conclusion}

We do not address the global open problem of non-trivial solutions to
$a^k+b^k=c^k+d^k$ for $k\ge5$ without linear constraints.
Our contribution is structural in the sliced setting~\eqref{eq:shift}.

On the central slice $a+b=c+d$ we have complete uniqueness (Theorem~\ref{thm:central-uniqueness}) and
quantitative separation (Theorem~\ref{thm:quantitative}): distinct unordered pairs $\{a,b\}$ along
$a+b=S$ produce values of $a^k+b^k$ separated by $\gg_k S^{k-2}$.
This gives a simple local instance of sparsity heuristics for equal sums of two $k$th powers.

For shifted slices $h\neq0$, MDO (Theorem~\ref{thm:MDO}) produces a \emph{squarefree} modulus $M_k$
with $M_k\mid h$, maximal among such moduli by Lemma~\ref{lem:maximality}. This contains parity
as a special case (Corollary~\ref{cor:parity}), yields $k\mid h$ for prime $k$, and gives
density $1/M_k$ of admissible shifts (Corollary~\ref{cor:density}).
For many odd $k$ (e.g.\ $k=13$) this drastically thins the set of feasible $h$.
Combined with the exclusion zone (Theorem~\ref{thm:exclusion}) we obtain the lower bound
$\min\{S,S+h\} \ge 2M_k/(k-1)$ (Proposition~\ref{prop:combination}), which is an explicit structural constraint on the
slice size. In addition, the overlap bound of Lemma~\ref{lem:overlap} shows that the smaller of the two slice sums must satisfy
$\min\{S,S+h\}\gg_k |h|$, giving a global geometric restriction based solely on the ranges of $a^k+b^k$ on the two slices.
Finally, along a fixed slice $(S,h)$, the dominance bound $k\le \max\{S,S+h\}\log 2$
(Theorem~\ref{thm:asymptotic}) shows that only finitely many exponents can occur.

From a broader perspective, one may compare these elementary constraints with global spacing heuristics
and conjectural arithmetic--geometric uniformity (e.g.\ Bombieri--Lang), cf.\ \cite{Browning2002,CaporasoHarrisMazur1997,Faltings1983,Lang1983}.

\end{document}